\documentclass[12pt,a4paper]{amsart}
\usepackage[a4paper]{geometry}
\newtheorem*{theorem}{Theorem}
\newtheorem{lemma}{Lemma}

\theoremstyle{remark}
\newtheorem*{remark}{Remark}

\newcommand{\set}[2]{\ensuremath{\{ #1 \>|\> #2 \}}}

\def\liebrack{\ensuremath{[\,\cdot\, , \cdot\,]}}      
\def\curlybrack{\ensuremath{\{\,\cdot \,, \cdot\,\}}}  

\begin{document}

\title[A Lie algebra that can be written as a sum of two nilpotent subalgebras]
{A Lie algebra that can be written as a sum of two nilpotent subalgebras, is solvable}
\author{P.A. Zusmanovich}
\address{}
\email{justpasha@gmail.com}
\date{last minor revision July 8, 2015}
\thanks{Mat. Zametki \textbf{50} (1991), $\mathcal{N}$3, 38--43 (in Russian); 
Math. Notes \textbf{50} (1991), 909--912 (English translation); 
\textsf{arXiv:0911.5418}.}
\thanks{}

\maketitle

In 1963 O. Kegel raised the following question: is a Lie ring written as a sum of 
two nilpotent subrings solvable? Recently, Kostrikin \cite{1} brought a renewed 
attention to this question in the case of finite-dimensional algebras over a field.
The question is easily solved in the affirmative in the case of characteristic zero
(cf. \cite{1} or \cite{2}). The purpose of this note is to prove the following 
theorem. 

\begin{theorem}
Over a field of characteristic $p>5$, a finite-dimensional Lie algebra written as a
sum of two nilpotent subalgebras, is solvable.
\end{theorem}

In \cite{1}, \cite{3}, \cite{4}, \cite{5}, a similar statement is proved under 
additional restrictions on the nilpotency index of one summand (with fewer 
restrictions on the characteristic of the ground field). Note that the theorem is 
no longer true when $p = 2$ (an appropriate counter example has been constructed in
\cite{4}). We make an essential use of Weisfeiler's results \cite{6} on Lie 
algebras with a solvable maximal subalgebra which dictates a restriction upon the 
characteristic of the ground field. 

We turn to the proof of the theorem. We may assume the ground field $K$ to be 
algebraically closed. Let $L$ be a counter example to the theorem having the least 
possible dimension, let $L = A + B$ be the corresponding decomposition into a sum 
of nilpotent subalgebras, where $\dim A \le \dim B$. The following lemma can be 
easily deduced from the minimality of the counter example (cf. \cite{1}). 

\begin{lemma}\label{lemma-1}\hfill
\begin{enumerate}
\item $L$ is semisimple.
\item If $L_0$ is a proper subalgebra of $L$ containing $A$ (or $B$), then $L_0$ is 
solvable. In particular, $L$ possesses a solvable maximal subalgebra.
\end{enumerate}
\end{lemma}

The following result has been proved in \cite{6}: a maximal solvable subalgebra in 
a simple Lie algebra determines a long filtration in it. A closer analysis of this 
proof enables us to replace the simplicity condition by semisimplicity. Indeed, a 
semisimple Lie algebra with a solvable maximal subalgebra possesses a unique ideal 
by Block's theorem \cite{7}. The proof of Theorem 1.2.2 in \cite{6} may be repeated
for such algebras almost verbatim. But the proof of Theorem 1.5.1 relies only on 
the conclusion of Theorem 1.2.2 and never makes use of the simplicity of the 
algebra. 

Let $L_0$ be a maximal solvable subalgebra of $L$ containing $A$ (it exists due to 
Lemma 1). By the above, we can apply Theorem 2.1.3 in \cite{6} which states that 
for the filtration $L \supset \dots \supset L_{-1} \supset L_0 \supset L_1 \supset \dots$,
determined by the subalgebra $L_0$, the associated graded algebra will have the form
$$
gr L = G = S \otimes O_m + D ,
$$
where $S$ is a simple Lie algebra isomorphic to $sl_2(K)$ or to the Zassenhaus 
algebra $W_1(n)$, $O_m$ is the algebra of truncated polynomials in $m$ variables, $D$
is a solvable subalgebra of $W_m = Der(O_m)$ such that $O_m$ contains no 
$D$-invariant ideals. The grading is given as follows:
\begin{align}\label{1}
&G_i = \langle e_i \rangle \otimes O_m , \quad i\ne 0 \\
&G_0 = \langle e_0 \rangle + D ,  \notag
\end{align}
where $S = \bigoplus_{i\ge -1} \langle e_i \rangle$ is the standard grading. 
Henceforth, this particular grading is meant each time we refer to 
$S \otimes O_m + D$ as a graded algebra. 

If $D = 0$, then $L$ is simple and, by Corollary 2.1.4 in \cite{6}, $L = S$. The 
impossibility of representing it as a sum of two nilpotent subalgebras can be 
easily verified in this case, for instance, by means of an argument which we will 
apply below in the general case. So we assume henceforth that $D \ne 0$. 
Furthermore, in the case of $S = sl_2(K)$ the remaining argument is either the same
as that for $S = W_1(n)$, or much simpler. Thus, we put henceforth $S = W_1(n)$. 

It has been shown in \cite{8} that each $\{G_i\}$-deformation of the algebra $G$ 
determined by the grading (\ref{1}) contains an ideal which is a deformation of the
algebra $W_1(n) \otimes O_m$. We will provide a cohomological proof of a somewhat 
more general result using definitions, notation, and facts from \cite{9}. 

We can write
$$
\{x,y\} = [x,y] + \sum_{s\ge 1} \psi_s(x,y) = [x,y] + \psi(x,y) ,
$$
where $\liebrack$ and $\curlybrack$ are multiplications in the algebras $G$ and 
$L$, respectively, and
$$
\psi_s \in C_s^2(G,G) = \set{\psi\in C^2(G,G)}{\psi(G_i,G_j) \subseteq G_{i+j+s}} .
$$

The Jacobi identity implies that
\begin{equation}\label{2}
d\psi_s + \sum_{i+j = s} \psi_i * \psi_j = 0 ,
\end{equation}
where $d$ is the coboundary operator, and $*$ is defined as follows:
$$
\varphi * \psi (x,y,z) = 
\varphi(\psi((x,y),z) + \varphi(\psi(z,x),y) + \varphi(\psi(y,z),x) .
$$

We want to show that, up to coboundaries, $\psi_s = 0$ for $1\le s < p$. Suppose 
that, by means of a coboundary change, we have already achieved the equalities 
$\psi_s = 0$, $1\le s < k < p$. By virtue of (\ref{2}), $\psi_k \in Z_k^2(G,G)$. 
Defining the action of $G$ on $C^2(G,G)$ in the standard way, it is easy to see that
$Z_k^2(G,G)$ is invariant relative to the action of $G_0$. 

Consider the torus $e_0 \otimes \langle 1 \rangle$ in $G$. The root subspaces 
relative to the action of this torus in $G$ are
\begin{align*}
&\widehat{G}_i = 
\langle e_i, e_{i+p}, e_{i+2p}, \dots, e_{i+p^n-p} \rangle \otimes O_m, 
\quad i\in \mathbb Z_p, i\ne 0, \\
&\widehat{G}_0 = 
\langle e_0, e_p, e_{2p}, \dots, e_{p^n-p} \rangle \otimes O_m + D .
\end{align*}
By the above, each cocycle in $Z_k^2(G,G)$ is cohomologically equivalent to some 
cocycle in $Z_k^2(G,G)$ invariant relative to the action of this torus, and we may 
put
$$
\psi_k(\widehat{G}_i, \widehat{G}_j) \subseteq \widehat{G}_{i+j} .
$$
But since $\psi_k\in Z_k^2(G,G)$ and $k < p$, we have $\psi_k = 0$.

Thus, 
\begin{equation}\label{3}
\psi_k(G_i,G_j) \subseteq \bigoplus_{s \ge p} G_{i+j+s} .
\end{equation}

We will view the algebra $L$ as the vector space $W_1(n) \otimes O_m + D$ with 
multiplication $\curlybrack$. On the subalgebra $B$ an induced filtration can be 
constructed: $B_i = B \cap L_i$. We define a homomorphism
$$
\varphi: gr B \to gr L, \quad x + B_{i+1} \mapsto x + L_{i+1}, \quad x\in B_i .
$$
It is easily seen that $Ker\,\varphi = 0$ and, therefore, we may identify $gr B$ 
with a subalgebra of $gr L$.

\begin{lemma}\label{lemma-2}\hfill
\begin{enumerate}
\item $gr B$ is a homogeneous nilpotent subalgebra of $gr L$.
\item $(gr B)_{-1} = e_{-1} \otimes O_m$.
\item 
$pr_D gr B = D$, where the left-hand side of the equality denotes the 
projection of $gr B$ onto $D$ (in the algebra $gr L$).
\end{enumerate}
\end{lemma}

\begin{proof}
(i) is obvious.

(ii) Since $L = A + B$ and $A \subseteq L_0$, we have
$$
(gr B)_{-1} = (gr L)_{-1} = e_{-1} \otimes O_m .
$$

(iii) Clearly, $pr_D gr B$ is a subalgebra of $D$. Therefore, 
$W_1(n) \otimes O_m + pr_D gr B$ is a subalgebra of $W_1(n) \otimes O_m + D$ 
containing $gr B$ and closed under the multiplication $\curlybrack$ (this 
follows from (\ref{3})). So there exists a subalgebra $M$ of $L$ such that 
$B \subseteq M \subseteq L$ and $gr M = W_1(n) \otimes O_m + pr_D gr B$. If 
$pr_D gr B \ne D$, then $M \ne L$ and, by Lemma \ref{lemma-1}, $M$ is solvable, 
whence $gr M$ is solvable, which is impossible.
\end{proof}

\begin{lemma}\label{lemma-3}
Suppose that $N$ is a subalgebra of $W_1(n) \otimes O_m + D$ (relative to the usual
multiplication $\liebrack$), and $O_m$ contains no proper $D$-invariant ideals.
If $N$ satisfies the conclusions of Lemma \ref{lemma-2}, then
\begin{enumerate}
\item $N_0 \simeq D$.
\item $N \subseteq \langle e_{-1}, e_0 \rangle \otimes O_m + D$.
\item $D$ consists of nilpotent (viewed as derivations of $O_m$) elements.
\end{enumerate}
\end{lemma}

\begin{proof}
Put $F = \set{f\in O_m}{e_0 \otimes f \in N}$. We choose an arbitrary element 
$d \in D$ and find $g\in O_m$ such that $d + e_0\otimes g \in N$. Then for each 
$f\in F$
$$
e_0 \otimes d(f) = [e_0 \otimes f, d + e_0 \otimes g] \in N .
$$
Therefore, $D(F) \subseteq F$. But then $FO_m$ is a $D$-invariant ideal of $O_m$, 
whence $FO_m = 0$ or $O_m$. Thus, either $F = 0$, or $F$ contains a polynomial $f$ 
with a nonzero constant term. In the latter case the equality 
$$
ad(e_0 \otimes f)^p (e_{-1} \otimes 1) = e_{-1} \otimes f^p = e_{-1} \otimes 1
$$
leads to a contradiction with the nilpotency of $N$. Thus $F = 0$.
The map $e_0 \otimes f + d \mapsto d$ is the isomorphism required in (i). 

Let $e_i \otimes g \in N$, $i>0$. Commuting this element as many times as necessary
with $e_{-1} \otimes 1$, we obtain the element $e_0 \otimes g \in N$, whence 
$g = 0$, which proves (ii).

Choose again an arbitrary element $d\in D$, find $g\in O_m$ such that 
$d + e_0 \otimes \in N$, and note that the action of $ad\,(d + e_0 \otimes g)$ on
$e_{-1} \otimes o_m$ is determined by the action of the operator $d + R_g$ on $O_m$,
where $R_g$ is the multiplication by the element $g$ in $O_m$. The nilpotency of $N$
implies that $(d + R_g)^{p^k} = 0$ for some $k$. Jacobson's formula and the last 
equality imply that $d^{p^k} = 0$, which proves (iii). 
\end{proof}

By Lemmas 2 and 3, $\dim B = \dim gr B \le p^m + \dim D$. According to our initial
assumption, $\dim L \le 2 \dim B$, whence $p^{n+m} + dim D < 2p^m + 2\dim D$ and 
$\dim D \ge p^m$. The proof of the theorem is concluded by 

\begin{lemma}\label{lemma-4}
Suppose $D$ is a subalgebra of $W_m$ consisting of nilpotent elements, and $O_m$ 
contains no $D$-invariant ideals. Then $\dim D < p^m$.
\end{lemma}

\begin{proof}
We will perform induction on $m$. For $m = 1$ the statement of the lemma is obvious
(each such subalgebra is one-dimensional). Put $Z = \set{z\in Z(D)}{z^p = 0}$. We 
have
$$
D(Z(O_m)) \subseteq Z(D(O_m)) \subseteq Z(O_m) .
$$
Arguing like in the proof of Lemma 3, we deduce that either $Z(O_m) = 0$ or $Z(O_m)$
contains a polynomial with a nonzero constant term. But the former is impossible 
because $D$ consists of nilpotent elements, so $Z \ne 0$. Therefore, 
$Z \not\subset (W_m)_0$, where $(W_m)_0$ is the zeroth term in the standard 
filtration (otherwise $Z(O_m)$ is contained in the maximal ideal of $O_m$). Choose 
$z\in Z$, $z \notin (W_m)_0$. It follows from Demushkin's results \cite{10} that $z$
is conjugate to the element $\partial/\partial x_1$ in $W_m$. Therefore, we may 
assume that
$$
D \subseteq C_{W_m}(\partial/\partial x_1) = 
\langle f\partial/\partial x_1 | f \in O[x_2, \dots, x_n] \rangle + 
W_{m-1}(x_2, \dots, x_m) .
$$
Here $W_{m-1}(x_2, \dots, x_m)$ is the Lie algebra of derivations of the algebra 
$O[x_2, \dots, x_m]$ of truncated polynomials in the variables $x_2, \dots, x_m$. 
It is easily seen that the first term is an ideal in this centralizer, so 
$pr_{W_{m-1}} D$ is a subalgebra of $W_{m-1}$ consisting of nilpotent elements. If 
$I$ is a $pr_{W_{m-1}} D$-invariant ideal in $O[x_2, \dots, x_m]$, then $(x_1)I$ is
a $D$-invariant ideal in $O_m$. Therefore, $I$ is a trivial ideal and we may apply
the induction hypothesis:
$$
\dim D \le p^{m-1} + \dim pr_{W_{m-1}} D < p^{m-1} + p^{m-1} < p^m .
$$
\end{proof}

\begin{remark}
Of course, Lemma \ref{lemma-4} is far from being the best result in this direction,
but, apparently, it suffices for our purposes.
\end{remark}

In conclusion, let us make several remarks concerning the possibility of obtaining the 
converse to the theorem. As has been noted in \cite{2}, the equality $L = H + L^2$,
where $H$ is the Cartan subalgebra of a Lie algebra $L$, provides a decomposition 
into a sum of two nilpotent subalgebras for Lie algebras with a nilpotent 
commutant. As has been shown in \cite{11}, this class of algebras coincides with 
the class of supersolvable Lie algebras. The semidirect sum $L + V$, where $L$ is a
nilpotent Lie algebra and $V$ is a faithful irreducible $L$-module, provides an 
example of a nonsupersolvable Lie algebra with such a decomposition. On
the other hand, if $L$ is the two-dimensional non-abelian Lie algebra and $V$ is the
irreducible $p$-dimensional $L$-module, then we obtain an example of a solvable Lie
algebra for which such a decomposition is impossible. Thus, the class of Lie 
algebras that can be represented as a sum of two nilpotent subalgebras, 
contains the class of supersolvable algebras and is contained in the class of solvable algebras,
both inclusions being strict. It would be interesting to provide a description of 
this class. 

The author is grateful to A.S. Dzhumadil'daev for his attention and help in this 
work.

\renewcommand{\refname}{Literature cited}

\end{document}